%% file: group-actions-on-segal-operads.tex
\documentclass{amsart}
\usepackage[matrix,cmtip,arrow,curve]{xy}
\usepackage[unicode]{hyperref}
\usepackage{graphicx}
\usepackage{color}
\usepackage{amssymb}

\numberwithin{equation}{section}

\title{Group actions on Segal operads}
\author[Julia E. Bergner]{Julia E. Bergner}\thanks{The first-named author was partially supported by NSF grants DMS-0805951 and DMS-1105766, and by a UCR Regents Fellowship.}
\address{Department of Mathematics, University of California, Riverside}

\email{bergnerj@member.ams.org}
\author{Philip Hackney}
\email{hackney@math.ucr.edu}

\date{\today}



\theoremstyle{plain}
\newtheorem{thm}[equation]{Theorem}
\newtheorem{lem}[equation]{Lemma}
\newtheorem{cor}[equation]{Corollary}
\newtheorem{prop}[equation]{Proposition}
\theoremstyle{definition}
\newtheorem{defn}[equation]{Definition}
\newtheorem{rem}[equation]{Remark}
\newtheorem{example}[equation]{Example}

\newcommand{\Hom}{\operatorname{Hom}}
\newcommand{\N}{\mathbb{N}}
\newcommand{\ob}{\operatorname{Ob}}
\newcommand{\mor}{\operatorname{Mor}}
\newcommand{\mcc}{\mathcal{C}}
\newcommand{\mct}{\mathcal{T}}
\newcommand{\id}{\operatorname{id}}
\newcommand{\set}[1]{\{#1\}}
\newcommand{\setm}[2]{\{\,#1\mid#2\,\}}
\newcommand{\ci}{\subset}

\newcommand{\inv}{{-1}}
\newcommand{\colim}{\operatorname{colim}}

\newcommand{\del}{\mathbf{\Delta}}
\newcommand{\delo}{\del^{op}}
\newcommand{\om}{\mathbf{\Omega}}
\newcommand{\omop}{{\om^{op}}}
\newcommand{\Set}{{\mathcal{S}et}}
\newcommand{\sset}{{s\mathcal{S}et}}
\newcommand{\dsset}{{d\sset}}
\newcommand{\dset}{{d\Set}}
\newcommand{\dspace}{{\dsset}}
\newcommand{\oper}{{\mathcal{O}}}
\newcommand{\coper}{{\mathcal{C}\mathcal{O}}}
\newcommand{\nerve}{\operatorname{nerve}}
\newcommand{\prim}{\operatorname{prim}}
\newcommand{\hocolim}{\operatorname{hocolim}}
\newcommand{\M}{\mathcal{M}}
\newcommand{\el}{\mathcal{L}}
\newcommand{\class}{\mathbf{C}}
\newcommand{\map}{\operatorname{Map}}
\newcommand{\maph}{\operatorname{Map}^h}
\newcommand{\T}{\mathcal{T}}
\newcommand{\inclo}{\delta} 
\newcommand{\newPsi}{\Upsilon}

\newcommand{\wprim}{\widetilde{\prim}}
\newcommand{\Aut}{\operatorname{Aut}}

\newcommand{\Sets}{\Set}

\newcommand{\nsgo}{\mathcal{G}\oper}

\newcommand{\Tgop}{\mathcal T_{\nsgo}}
\newcommand{\SSets}{\sset}
\newcommand{\SSp}{\mathcal{SS}p}
\newcommand{\LSSets}{\el\sset}

\begin{document}

\begin{abstract}
We give a Quillen equivalence between model structures for simplicial operads, described via the theory of operads, and Segal operads, thought of as certain reduced dendroidal spaces.  We then extend this result to give a Quillen equivalence between the model structures for simplicial operads equipped with a group action and the corresponding Segal operads.
\end{abstract}

\maketitle

\tableofcontents

\section{Introduction}

There has been much recent interest in homotopical approaches to categories, and due to the development of dendroidal sets by Moerdijk and Weiss in \cite{mw}, some of this work is being extended to give homotopical approaches to multicategories, better known as colored operads.  Rather than simplicial categories, or categories enriched in simplicial sets, one can instead consider colored operads in simplicial sets.  Other models for so-called $(\infty, 1)$-categories are described by simplicial diagrams of either sets or spaces; replacing the simplicial indexing category $\del$ with the dendroidal category $\om$ results in models for $(\infty,1)$-operads.  These models and the comparisons between them, in the form of Quillen equivalences of model categories, are being developed in work of Cisinski and Moerdijk \cite{cm-ds}, \cite{cm-simpop}, \cite{cm-ho}.

In the world of $(\infty,1)$-categories, a beginning step in the comparison between simplicial categories and Segal categories, which can be thought of as up-to-homotopy nerves of simplicial categories, was the single-object case: the comparison between simplicial monoids and Segal monoids.  This result of the first-named author was extended to the case of simplicial categories and Segal categories with a given fixed object set, which then played a role in the proof of the general case \cite{simpmon}, \cite{thesis}.  The idea behind the proof was that simplicial monoids can be regarded as product-preserving functors from the algebraic theory of monoids $\T_M$ to the category of simplicial sets; finding an equivalence with certain kinds of diagrams from $\delo$ to the category of simplicial sets was, in a sense, replacing $\T_M$ with a simpler diagram that served the same purpose.

In the general case of simplicial colored operads and Segal colored operads, the equivalence given by Cisinski and Moerdijk did not use the same methods.  However, it is still true that ordinary (single-colored) simplicial operads can be regarded as product-preserving functors from an appropriate theory to the category of simplicial sets, and likewise for the category of simplicial operads with a fixed color set \cite{multisort}.  The first main result of this paper is to give an explicit Quillen equivalence in this case, making use of the machinery of algebraic theories.

\begin{thm}
The model category structure for Segal operads is Quillen equivalent to the model category of simplicial operads.
\end{thm}

This result was already proved in \cite{cm-simpop} by restricting the general case, but our use of the algebraic theory framework gives a different approach.

The second main result, whose proof follows from the same methods, is a comparison of simplicial operads with a group action and Segal operads with the analogous group action.  There are two ways to approach this situation, and we give a proof for each one.  In the first, we consider simplicial operads equipped with an action of a fixed discrete group $G$.  In the second, we look at simplicial operads equipped with an action of a simplicial group, where the group in question can vary.

\begin{thm}
Let $G$ be a discrete group.  The model category structure for Segal operads with a $G$-action is Quillen equivalent to the model category of simplicial operads with a $G$-action.
\end{thm}

\begin{thm}
The model category structure for Segal operads with a simplicial group action is Quillen equivalent to the model category of simplicial operads with a simplicial group action.
\end{thm}

After a background section, we establish some model category structures for reduced dendroidal spaces.  We then prove that those which are local with respect to a Segal condition are equivalent to simplicial operads, viewed as algebras over the theory of operads.  From there, we consider the two different theories for groups acting on operads and find simpler diagrams, as developed in \cite{bergnerhackney}, which we prove encodes the same structure.  The paper concludes with technical results which are used in earlier sections.

\section{Background}

In this section we give a brief review of some of the main objects of study in this paper, namely dendroidal sets and algebras over multi-sorted algebraic theories, as well as the model category framework in which we investigate them.

Recall that an \emph{operad} in simplicial sets is a sequence $\{P(k)\}_{k \geq 0}$ of simplicial sets, each with a right action of the symmetric group $\Sigma_k$, a
unit element $1 \in P(1)$, and operations
\[ P(k) \times P(j_1) \times \cdots \times P(j_k) \rightarrow P(j_1+ \cdots +j_k) \] satisfying
associativity, unit, and equivariance conditions \cite{gils}.  More generally, one can look at \emph{colored operads}, in which inputs and outputs for each operation have designated labels, and composition respects these labels \cite{bm-resolution}, \cite{bv}.
In fact, a colored operad in simplicial sets is the same as a symmetric simplicial \emph{multicategory}, in which morphisms have a single target but possibly multiple (or no) inputs.

We reserve the word ``operad" to mean a colored operad with a single color, i.e., an ordinary operad, and write $\oper$ for the category of operads. We write $\coper$ for the category of colored operads.

\subsection{Dendroidal sets}

Recall that $\del$ denotes the category of finite ordered sets and order-preserving maps between them.  There is a close relationship between simplicial sets, or functors $\delo \rightarrow \Set$, and categories.  We can think of the objects $[n]$ of $\del$ as categories with $n+1$ objects and a single morphism $i \rightarrow j$ for $i \leq j$.  The nerve of a category $\mathcal C$ is defined to be the simplicial set whose $n$-simplices are given by the set $\Hom([n], \mathcal C)$.

The idea behind dendroidal sets is to find a generalization of the category $\del$ suitable for working with colored operads rather than categories.  The objects of this category are given by certain kinds of trees.

By ``tree", we mean a non-planar graph with no cycles, where some edges may touch only one vertex.  A tree must have at least one such edge, specified as the \emph{root}; any others the tree might have are called \emph{leaves}. The data of the tree consists of both the graph and the choice of root.  As an example, for $n\geq 0$, the \emph{corolla} $C_n$ is the tree with one vertex and $n+1$ edges, one of which is specified as the root.

The root determines a direction to the tree, where the leaves are regarded as inputs and the root is regarded as the output.  Consequently, every vertex $v$ has \emph{incoming edges}, whose number is given by $|v|$, the \emph{valence} of the vertex, as well as one \emph{outgoing edge}, which is the one in the direction of the root.
For a tree $S$, the set of vertices $V(S)$ is actually a \emph{set with valences}, namely the set $V(S)$ together with a function $V(S) \rightarrow \N$.
We let $[n]$ denote the linear tree with $n$ vertices and $n+1$ edges.

A tree $S$ determines a free colored operad as follows.  The set of colors is just the set of edges $E(S)$, and we have one generator for each vertex $v\in V(S)$. Choosing an order for the incoming edges $e_1, \dots, e_n$, the generator $v$ is in $(e_1, \dots, e_n; e)$ where $e$ is the outgoing edge of $v$.
We call this colored operad $S$ as well; the isomorphism class does not depend on our choice of planar structure.
(This colored operad is called $\Omega(S)$ in \cite{mw}.)

Notice that the colored operad $S$ is actually quite small despite being free. The fact that each edge has a distinct color forces all of the generated compositions to come from contracting internal edges of $S$.

The category $\om$, defined as the full subcategory of the category of colored operads with objects the free colored operads generated by trees, is the desired analogue of $\del$ for operads.
As for $\del$, we assume that we are working in some chosen skeletal subcategory of $\om$; in particular, we always regard $\ob \om$ as a set.

Notice that if $f: R \to S$ is a morphism of $\om$, and $S$ is linear, then $R$ is linear as well. In particular, if we regard the simplicial indexing category $\del$ as a subcategory of $\om$, we have $\Hom_\om ([m], [n]) = \Hom_\del([m], [n])$. Denoting by $\eta$ the tree with one edge and no vertices, it follows that the overcategory $\om / \eta$ is equivalent to $\del$.

\begin{defn} A \emph{dendroidal set} is a functor $\omop \to \Set$. More generally, if $\mcc$ is a category then a \emph{dendroidal object in $\mcc$} is a  functor $\omop \to \mcc$.
If $X$ is a dendroidal set and $S$ is a tree, we will write $X_S$ for the evaluation of $X$ at $S$.
\end{defn}

We denote by $\dset$ the category of dendroidal sets.  For a tree $S$, we consider the representable dendroidal set $\Omega[S] = \Hom_\om (- , S)$.  In this paper, we are primarily interested in \emph{dendroidal spaces}, or functors $\omop \rightarrow \SSets$.  We denote by $\dsset$ the category of dendroidal spaces.  Notice that any simplicial set or dendroidal set can be regarded as a dendroidal space by taking it to be constant in the dendroidal or simplicial direction, respectively; in 
this way
we consider $\Omega[S]$ as a
discrete
dendroidal space.

\begin{defn}
The \emph{nerve of a colored operad $P$} is the dendroidal set given by
\[ \nerve(P)_S = \Hom_{\coper}(S, P). \]
\end{defn}
Notice that $\nerve(S) = \Omega[S]$.

\subsection{Model categories}

Throughout this paper we use the framework of model categories; standard references are \cite{hirschhorn}, \cite{hovey}, and the original \cite{quillen}.  We make use of the standard model structure on the category $\sset$ of simplicial sets \cite[I.11]{gj}.  In particular, we are concerned with model categories whose objects are simplicial presheaves on some diagram category $\mathcal C$, or functors $\mathcal C^{op} \rightarrow \sset$.  Such a category is denoted by $\sset^{\mathcal C^{op}}$.

For such a category, we always have the \emph{projective} model structure, where weak equivalences and fibrations are given by levelwise weak equivalences and fibrations of simplicial sets \cite[11.6.1]{hirschhorn}.  If $\mathcal C$ has the structure of a Reedy category, then $\sset^{\mathcal C^{op}}$ additionally has the \emph{Reedy} model structure, where the weak equivalences are defined levelwise \cite[15.3.4]{hirschhorn}.  For example, the category $\del$ is a Reedy category, so the category of simplicial spaces has the Reedy model structure.  The category $\om$ is not a Reedy category, but is a \emph{generalized Reedy category} in the sense of Berger and Moerdijk \cite{bm}, and $\sset^{\omop}$ has a corresponding generalized Reedy model structure \cite[1.6]{bm}.

We use several localized model structures in this paper. The (left) Bousfield localization of a model category $\M$ with respect to a class of maps $\class$ is a modified model structure $\mathcal{L}_\class \M$ on the underlying category of $\M$ which has the same class of cofibrations and the maps in $\class$ become weak equivalences.

To be more precise, we make use of the homotopy function complex $\maph(A,X)$ (see \cite[Ch.\ 17]{hirschhorn} or \cite{dwyerkan}).  When $A$ is cofibrant, $X$ is fibrant, and $\M$ is a simplicial model category, the homotopy function complex agrees with the usual mapping space.
An object $W$ is called \emph{$\class$-local} if it is fibrant in $\M$ and for every element $f\colon A \to B$ of $\class$ the induced map of homotopy function complexes $f^*\colon \maph(B,W) \to \maph(A,W)$ is a weak equivalence; the $\class$-local objects are precisely the fibrant objects of $\mathcal{L}_\class \M$.
A map $g\colon X\to Y$ in $\M$ is a \emph{$\class$-local equivalence} if for every $\class$-local object $W$ the induced map of homotopy function complexes $g^*\colon \maph(Y,W) \to \maph(X,W)$ is a weak equivalence; the $\class$-local equivalences are precisely the weak equivalences of $\mathcal{L}_\class\M$.
More details about Bousfield localization, including existence proofs, may be found in \cite[Ch.\ 3 and 4]{hirschhorn}.

\subsection{Algebraic theories}

In this section we recall the definition of an algebraic theory and describe the theory of operads.

\begin{defn}
Given a set $\mathbf S$, an $\mathbf S$-\emph{sorted algebraic theory} (or
\emph{multi-sorted theory}) $\mathcal T$ is a small category with
objects $T_{M}$ where $M = \left< m_1, \ldots
,m_n\right>$ for $m_i \in \mathbf S$ and $n \geq 0$ varying, and such
that each $T_{M}$ is equipped with an isomorphism
\begin{equation} T_{M} \cong \prod_{i=1}^n T_{m_i}. \label{E:theoryprod}\end{equation}
For a particular $M$, the entries $m_i$ can repeat,
but they are not ordered.
In other words, $M$ is a an
$n$-element subset with multiplicities.
There exists a terminal
object $T_0$ (corresponding to the empty subset of $\mathbf S$).
\end{defn}

Let $\mathcal C$ be a category with coproducts such that given any
element $m \in \mathbf S$, we have a forgetful functor
\[ \varphi_m : \mathcal C \rightarrow \Set \]
and its left adjoint, the free functor
\[ \lambda_m : \Set \rightarrow \mathcal C. \]
Given this structure, we can construct the
$\mathbf S$-sorted theory corresponding to the category $\mathcal C$.  Its objects are the finitely generated free objects; morphisms are the opposites of those in $\mathcal C$.

Details on the conditions required for $\mathcal C$ to have an associated multi-sorted theory, and an explanation of how the operations are developed in such a theory, can be found in \cite{multisort}.

\begin{example}
In the category of groups, consider the full subcategory of representatives of isomorphism classes of finitely generated free groups.  The opposite of this category is the theory of groups, denoted by $\mathcal T_G$.  Since the objects of $\mathcal T_G$ are obtained as finite products of the free group on one generator, the theory $\mathcal T_G$ is ``1-sorted" or an ``ordinary" algebraic theory.
\end{example}

\begin{example}{\cite[3.4]{multisort}} \label{operad}
The most important example of a multi-sorted theory for the purposes of this paper is the $\mathbb N$-sorted
theory of symmetric operads. Here we consider operads in the category of sets.

There is a notion of a free operad on $n$ generators in arities
$m_1, \ldots ,m_n$ \cite[ II.1.9]{mss}, \cite[2.3.6]{rezkop}.
Specifically, such a free operad has, for each $1 \leq i \leq n$,
a generator in $P(m_i)$.  Note that the values of $m_i$ can
repeat. For example, one can think of the free operad on $n$
generators, each of arity 1, as the free monoid on $n$ generators.

In the category of operads $\oper$, consider the full subcategory of representatives of
isomorphism classes of finitely generated free operads.  Each object in this
category, then, can be described as the free operad on $n$
generators of arities $m_1, \ldots ,m_n$ for some $n, m_1, \ldots ,m_n \geq 0$. The opposite of this category is the theory of
operads $\T_\oper$. Using the notation for multi-sorted
theories, we have that $T_m$ for $m \in \mathbb N$ is
the free operad on one generator at level $m$ and for
$M = \left< m_1, \ldots ,m_n\right>$, we have that
$T_{M}$ is the free operad on $n$ generators at levels
$m_1, \ldots m_n$.
\end{example}

\subsection{Algebras over multi-sorted theories}\label{SS:comparison}

Let $\T_\oper$ denote the multi-sorted theory of operads, as just described, and let $\star$ denote the trivial operad (which is initial in $\oper$ and terminal in $\T_\oper$).  Let $\sset$ denote the category of simplicial sets.  By $Alg^{\T_\oper}$ we denote the category of product-preserving functors $\T_\oper \rightarrow \sset$.  There is a model structure on $Alg^{\T_\oper}$ in which weak equivalences and fibrations are given levelwise \cite{badzioch}.

We want also to look at such functors in which products are preserved only up to homotopy, called \emph{homotopy} $\T_\oper$-\emph{algebras}.  By definition, such functors $W: \T_\oper \to \sset$ only satisfy the condition that $W_0$ be weakly contractible, but here we impose the additional condition that they be objects in the full subcategory $\sset_*^{\T_\oper}$ of $\sset^{\T_\oper}$, whose objects are the functors $W: \T_\oper \to \sset$ with $W(\star) = *$.

The following proposition holds more generally for other
algebraic theories, but we state it for our main example of interest; its proof is a modification of results found in \cite{badzioch} and \cite{multisort}.

\begin{prop} \cite[Proposition 2]{simpmoncorr}
There is a model category structure on $\sset_*^{\T_\oper}$, denoted $\el \sset_*^{\T_\oper}$, in which the fibrant objects are homotopy $\T_\oper$-algebras.  Furthermore, this model structure is Quillen equivalent to $Alg^{\T_\oper}$.
\end{prop}

This model structure is obtained by
\begin{enumerate}
\item considering $\sset_*^{\T_\oper}$ with the model structure where the weak equivalences and fibrations are precisely the levelwise weak equivalences and fibrations, and then
\item localizing with respect to the set of maps
\begin{equation} \left\{ \coprod \Hom(T_{m_i}, -) \to \Hom(T_M, -) \right\}_M \label{E:theorylocalize}\end{equation}
which are induced from the projections in \eqref{E:theoryprod}.
\end{enumerate}

\section{Model structures on categories of dendroidal spaces}

Regarding the category $\dspace$ as the category of diagrams, $\sset^\omop$, we denote by $\dspace_f$ the projective model structure, and by $\dspace_R$ the generalized Reedy structure.  We begin by recalling the following characterization of the cofibrations in $\dspace_R$.

\begin{defn}{\cite[Proposition 1.5]{cm-ho}} \label{D:normal}
A monomorphism $f: X\to Y$ in $\dspace$ is \emph{normal} if for any tree $S$, the action of $\Aut(S)$ on $Y_S \setminus X_S$ is free.  A dendroidal space $X$ is \emph{normal} if $\varnothing \rightarrow X$ is normal.
\end{defn}

\begin{thm}{\cite{cm-ds}} \label{T:cm-modelstruct}
The cofibrations in $\dspace_R$ are the normal monomorphisms.
\end{thm}

The generating cofibrations of the model structure $\dspace_f$ can be shown to be normal, and the weak equivalences are the same in both model structures, so $\dspace_f$ is Quillen equivalent to $\dspace_R$.  Because we need a similar result for generalized Reedy categories other than $\omop$, we prove the following more general result.

\begin{prop}
Let $\mathcal C$ be a generalized Reedy category.
The identity functor induces a Quillen equivalence
\[ \sset_f^{\mathcal C} \rightleftarrows \sset_R^{\mathcal C}. \]
\end{prop}

\begin{proof}
It suffices to prove that every Reedy fibration is a levelwise fibration.  For ordinary Reedy categories, the proof is given in \cite[15.3.11, 15.6.3]{hirschhorn}.  We modify the proof for the generalized case.

Recall from \cite{bm} that $f \colon X \rightarrow Y$ is a generalized Reedy fibration if, for every object $\alpha$ in $\mathcal C$, the relative matching map $X_\alpha \rightarrow M_\alpha \times_{M_\alpha(Y)} Y_\alpha$ is a fibration in $\sset^{\Aut(\alpha)}$. In other words, this map is a fibration in $\sset$ which is $\Aut(\alpha)$-equivariant.

So, suppose that $f \colon X \rightarrow Y$ is a generalized Reedy fibration.  Notice that the argument in \cite[15.3.9]{hirschhorn} still holds in this setting, working in $\sset^{\Aut(\alpha)}$ rather than $\sset$, establishing that $M_\alpha(X) \rightarrow M_\alpha(Y)$ is a fibration in $\sset^{\Aut(\alpha)}$.  Similarly, the proof of \cite[15.3.10]{hirschhorn} gives us that $X_\alpha \rightarrow Y_\alpha$ is a fibration in $\sset^{\Aut(\alpha)}$, which is by definition a fibration in $\sset$.
\end{proof}

As with diagrams given by $\T_\oper$, we can consider the full subcategory $\dspace_* \hookrightarrow \dspace$ consisting of those functors $X: \om^{op} \to \sset$ such that $X_\eta$ is a point. We call such $X$ \emph{reduced dendroidal spaces}.

\begin{rem}\label{R:basepoint}
If $X$ is reduced, then $X_{[n]}$ has a basepoint given by the unique map $[n] \to \eta=[0]$.
\end{rem}

\begin{rem}\label{R:reducednerve}
If $P$ is a colored operad, then $\nerve(P)$ is reduced if and only if $P$ has a single color, which is our primary case of interest.
\end{rem}

\begin{defn}
Define the \emph{reduction} of a dendroidal space by $X\mapsto X_\ast$, where $X_\ast$ is the pushout
\[ \xymatrix{ X_\eta \times \Omega[\eta] \ar@{->}[r] \ar@{->}[d] & \Omega[\eta] \ar@{->}[d] \\
X \ar@{->}[r] & X_* }\]
Here the top map is the projection $X_\eta \times \Omega[\eta]_S \to \Omega[\eta]_S$ for each tree $S$,
and the map on the left is the adjoint of the identity map $X_\eta \overset{\id}\to X_\eta = \Hom_\dspace(\Omega[\eta], X)$.
\end{defn}

In fact, the reduction functor $r\colon \dspace \to \dspace_*$ is left adjoint to the inclusion functor $U\colon \dspace_* \to \dspace$.  We use both notations, $X_*$ and $r(X)$ to denote the reduction of a dendroidal space $X$.  If $R$ is any nonlinear tree, then $r(X)_R = X_R$ since $\Omega[\eta]_R = \varnothing$.

Notice that the limit of a diagram $\mathcal{J} \to \dspace_*$ is actually in $\dspace_*$, not just in $\dspace$, since $U$ is a right adjoint. As one may expect, $U$ does not preserve colimits in general (although it happens to preserve coequalizers).  However, we do have the following result.

\begin{prop}
The category $\dspace_*$ is cocomplete.
\end{prop}

\begin{proof}
If $\set{X^a}_{a \in \mathcal{A}}$ is a set of objects of $\dspace_*$, then the coproduct is defined by
\[ \left( \coprod_{a \in \mathcal{A}} X^a \right)_R = \begin{cases} \coprod_{a \in \mathcal{A}} X^a_R & \text{if $R$ is a nonlinear tree} \\
\bigvee_{a \in \mathcal{A}} X^a_R &\text{if $R$ is a linear tree,} \end{cases} \]
together with the evident structure maps. Coequalizers are created in $\dspace$. Since $\dspace_*$ has all coproducts and coequalizers, it is cocomplete by the dual of \cite[V.2 Cor.\ 2]{maclane}.
\end{proof}

As noted in Remark~\ref{R:basepoint},
if $X$ is a reduced dendroidal space and $S=[m]$ is a linear tree, then $X_S$ has a natural basepoint.
There are no maps $S\to \eta$ when $S$ is nonlinear, so $X_S$ does not have a natural basepoint, and, in fact, may be empty.

\begin{defn}
Suppose that $X$ is a reduced dendroidal space and $K$ is a simplicial set. We define a dendroidal space $X \otimes K$, regarded as a diagram $\omop \times \delo \rightarrow \Set$, by
\[ (X \otimes K)_{S,n} =
\begin{cases} X_{S,n} \times K_n & \text{if $S$ is nonlinear} \\
X_{S,n} \wedge (K_n)_+ & \text{if $S=[m]$ is linear.}
\end{cases} \]
\end{defn}

\begin{prop}\label{P:identifyrtimes}
Suppose that $X\in \dspace_*$ and $K\in \sset$, each regarded as a dendroidal space, and let $Z = X\times K$. Then
\[ Z_* = X\otimes K. \] In other words, $X\otimes K = r(U(X)\otimes K)$.
\end{prop}

\begin{proof}
Recall that $Z_*$ is defined as the pushout
\[ \xymatrix{
Z_\eta \times \Omega[\eta] \ar@{->}[r] \ar@{->}[d] & \Omega[\eta] \ar@{->}[d] \\
Z \ar@{->}[r] & Z_*
}\]
in $\dspace$. We construct a map $Z_* \to X\otimes K$; there is already a map $Z \to X\otimes K$.
Note that since $\Omega[\eta]_R = \varnothing$ if $R$ is nonlinear, in this case there is no change.

Suppose that $Y\in \dspace$ and we have maps
\begin{align*} f: \Omega[\eta] &\to Y \\
g: Z &\to Y
\end{align*}
which agree on $Z_\eta \times \Omega[\eta]$. The maps $f$ and $g$ determine a map $X \otimes K \to Y$ as follows. At a nonlinear tree $R$, $(X\otimes K)_R = Z_R$, so the map is just defined by $g$. If $R=[m]$ is a linear tree, then
\[ (X \otimes K)_{[m]} = X_{[m]} \wedge K_+. \]
Define the map \[ X_{[m]} \wedge K_+ \to Y_{[m]} \]
by
\begin{align*}
x \wedge k &\mapsto g(x,k) \\
x \wedge * &\mapsto f(*).
\end{align*}
We need to see that the top assignment is well-defined, i.e. $*\wedge k \mapsto g(*,k) = f(*)$, which it is.

Thus we have a map $X\otimes K \to Y$ extending $f$ and $g$, which is the only possibility, so $X\otimes K$ is precisely $Z_*$.
\end{proof}

\begin{prop}\label{P:projmodel}
There is a model category structure on $\dspace_*$ in which the fibrations and weak equivalences are defined levelwise.
\end{prop}

We write $\dspace_{*,f}$ for this model structure.

\begin{proof}
We use the conditions of \cite[11.3.1]{hirschhorn} to establish this model structure.  We have proved that the category $\dspace_\ast$ is complete and cocomplete, and the two-out-of-three and retract axioms follow as usual.

To obtain sets of generating cofibrations and acyclic cofibrations, we apply an appropriate method of reduction to the generating sets for the projective model structure on $\sset^{\om^{op}}$, as follows.  Since fibrations are levelwise in the projective structure, generating cofibrations can be taken to be maps of the form
\[ \Omega [S] \times \partial \Delta [n] \rightarrow \Omega [S] \times  \Delta [n] \]
where $n \geq 0$ and $S$ is an object of $\om$.  Similarly, generating acyclic cofibrations can be taken to be maps of the form
\[ \Omega[S] \times V[n,k]   \rightarrow \Omega [S] \times  \Delta [n] \] where $n \geq 1$, $0 \leq k \leq n$, $S$ is an object of $\om$, and $V[n,k]$ is the $k$-horn of $\Delta[n]$.
Now, define the set of generating cofibrations
\[ I_f \colon = \{ \Omega [S]_* \otimes \partial \Delta [n] \rightarrow \Omega [S]_* \otimes \Delta [n] \mid n \geq 0, S \text{ in } \om\} \] and generating acyclic cofibrations
\[ J_f \colon = \{\Omega [S]_* \otimes V[n,k] \rightarrow \Omega [S]_* \otimes \Delta [n] \mid n \geq 1, 0 \leq k \leq n, S \text{ in } \om\} \] which can be seen to satisfy the small object argument, satisfying condition (1).

By definition, the $I_f$-injectives are precisely the acyclic fibrations, and the $J_f$-injectives are the fibrations, so conditions (3) and (4)(b) are satisfied.  Furthermore, the $J_f$-cofibrations are $I_f$-cofibrations and weak equivalences, using the model structure on $\sset$, establishing condition (2).
\end{proof}

We recall the following lemma about normal monomorphisms in the generalized Reedy structure.

\begin{lem} \label{L:subthings}{\cite[1.8]{cm-ho}}
If $R$ is normal then any monomorphism $S\to R$ is normal.
\end{lem}

We next show that reduction and tensor products preserve normal objects.

\begin{prop}\label{L:rrnormal}
If $X$ is normal and $K$ is a simplicial set, then
\begin{enumerate}
\item $X_*$ is normal, \label{E:xstarnormal}
\item $X\times K$ is normal, and \label{E:xknormal}
\item $X_* \otimes K$ is normal. \label{E:xtensornormal}
\end{enumerate}
In particular, if $S$ is a tree then $\Omega[S]_*$ is normal.
\end{prop}
\begin{proof}
The dendroidal space $X_*$ is obtained as a pushout
\[ \xymatrix{
\coprod_{E(S)} \Omega[\eta] \ar@{->}[d] \ar@{->}[r] & \Omega[\eta] \ar@{->}[d] \\
X \ar@{->}[r] & X_*
}\]
with the left hand map an inclusion of a sub-dendroidal space into a normal dendroidal space, hence is a cofibration by Lemma~\ref{L:subthings}.
Thus $\Omega[\eta] \to X_*$ is a cofibration as well, so we have that the composite $\varnothing \to \Omega[\eta] \to X_*$ is a cofibration.

For \eqref{E:xknormal}, notice that the following is a consequence of  Definition~\ref{D:normal}:  if $f: X\to Y$ is a normal monomorphism and $K$ is a simplicial set, then \[ f \times \id_K \colon X\times K \to Y\times K\] is normal.

Finally, \eqref{E:xtensornormal} follows from \eqref{E:xstarnormal} and \eqref{E:xknormal} since $X_* \otimes K = (X_* \times K)_*$ by Proposition \ref{P:identifyrtimes}.
\end{proof}

\begin{prop}\label{P:reedymodel}
There is a model category structure on $\dspace_*$ in which the cofibrations are the normal monomorphisms and the weak equivalences are defined levelwise.
\end{prop}

We write $\dspace_{*,R}$ for this model structure.

\begin{proof}
The generating cofibrations in the generalized Reedy model structure $\dspace_R$ can be taken to be those of the form
\[ \partial \Delta [n] \otimes \Omega [S] \cup \Delta[n] \otimes \partial \Omega [S] \rightarrow \Delta[n] \otimes \Omega [S] \]
where $n \geq 0$ and $S$ is an object of $\om$.  The generating cofibrations can similarly be taken to be those of the form
\[ V[n,k] \otimes \Omega [S] \cup \Delta[n] \otimes \partial \Omega [S] \rightarrow \Delta[n] \otimes \Omega [S] \]
for $n \geq 1$, $0 \leq k \leq n$, and $S$ an object of $\om$.

Using Lemma \ref{L:subthings}, we can reduce these maps so that they are in the category $\dspace_\ast$ and verify that these reductions are in fact normal monomorphisms.  Call these sets of reduced maps $I_R$ and $J_R$, respectively.  It is not hard to verify that $I_R$-cofibrations are precisely the normal monomorphisms in $\dspace_\ast$ and that the $J_R$-cofibrations are the acyclic cofibrations.  Since the fibrations are given by a lifting condition with respect to the acyclic cofibrations, we can see that the $I_R$-injectives are in fact acyclic fibrations and that the $J_R$-injectives are fibrations.  Applying \cite[11.3.1]{hirschhorn} gives the desired model structure.
\end{proof}

\section{Localization of the model structures on \texorpdfstring{$\dspace_*$}{dsSet*}}

We now establish localizations of the model structures from the previous section, so that the fibrant objects can be regarded as homotopy operads.

\begin{defn} \cite{cm-ds} \label{D:segalcore}
Let $S$ be a tree, $V$ its set of vertices, and, for $v\in V$, $C_v \ci S$  the subtree with a single vertex $v$ (so that $C_v \cong C_{|v|})$.
Then the \emph{Segal core} of $S$ is
\[ Sc[S] = \bigcup_{v\in V} \Omega[C_v] \ci \Omega[S]. \]
\end{defn}

Recall that we regard $\om$ as a skeletal category. Define the set of maps
\begin{equation} \class = \setm{ Sc[S]_* \hookrightarrow \Omega[S]_*}{S\in \ob \om}. \label{E:class} \end{equation}

\begin{prop}\label{P:blexist}
The Bousfield localizations $\el_\class \dspace_{*,R}$ and $\el_\class \dspace_{*,f}$ exist and have the same class of weak equivalences.
\end{prop}

\begin{proof} To prove existence, we apply \cite[4.1.1]{hirschhorn}. To do so, we need only verify that $\dspace_*$ is left proper for each of these model structures.
Let $A \rightarrow B$ be a cofibration and $B \leftarrow A \rightarrow C$ a pushout diagram $\dspace_*$.  We want to show that the map $C \rightarrow B \amalg_A C$ is a cofibration also.  If we take the pushout in $\dspace$, notice that $B \amalg_A C$ is still reduced.  Therefore, these two pushouts coincide.  Then the result follows from left properness of the unreduced category.

A similar argument to that at the beginning of \cite[\S 7]{thesis} shows that these localized model categories have the same class of weak equivalences.
\end{proof}

We generally omit $\class$ from the notation, and just write $\el \dspace_{*,R}$ and $\el \dspace_{*,f}$ for these localizations. We call the fibrant objects in $\el \dspace_{*,R}$ \emph{Segal operads}.

Notice that in this localization $\Omega[R]_*$ is weakly equivalent to $\Omega[S]_*$ whenever the trees $R$ and $S$ have the same set of sub-corollas, as follows.  Observe that
\[ Sc[R]_* = \coprod_{v\in V(R)} \Omega[C_{v}]_* \cong \coprod_{v\in V(R)} \Omega[C_{|v|}]_* \] so that
$Sc[R]_* \cong Sc[S]_*$, and we have weak equivalences
\[ \Omega[R]_* \overset{\simeq}\hookleftarrow Sc[R]_* \cong Sc[S]_* \overset\simeq\hookrightarrow \Omega[S]_*. \]

The following is a variant of \cite[4.2]{simpmon}, and its proof, which is technical, is deferred to Section~\ref{S:filtrprop}.

\begin{prop}\label{P:filtration}
Let $P$ be the free operad on the generating set $M=\set{x_1^{j_1}, \dots, x_m^{j_m}}$, where $x^p$ is in arity $p$. If $S$ is any tree whose list of sub-corollas is $C_{j_1}, \dots, C_{j_m}$, then
\[ L_1 \Omega[S]_* \simeq \nerve(P) \]
in the localized model structure $\el\dspace_{*,R}$, where $L_1$ denotes its fibrant replacement functor.
\end{prop}

\section{Comparison with reduced homotopy algebras over \texorpdfstring{$\mct_\oper$}{toh}}

In this section, we give an explicit Quillen equivalence between the model categories $\el \dspace_{*,f}$ and $\el \sset_*^{\T_\oper}$.  We make the comparison via a functor $J: \om^{op} \to \T_\oper$.

\begin{defn}
The functor $J^{op}: \om \to \T_{\oper}^{op}$ takes a tree $S$ to the free operad on $V(S) = \set{v_1, \dots, v_n}$.
\end{defn}

It can be shown that the precomposition functor $J^* \colon \el \sset^{\T_\oper} \rightarrow \el \dspace_f$ and its left Kan extension $J_!$ restrict to give an adjoint pair
\[ J_!: \dspace_* \rightleftarrows \sset^{\T_\oper}_* :\!J^* \] on the reduced categories.
This left Kan extension can be described explicitly on the reduction of representables.

\begin{prop}\label{P:lke}
The left Kan extension of $\Omega[S]_*$ along $J$ is  $\Hom_{\T_\oper} (J(S), -)$.
\end{prop}

We defer the proof to Section~\ref{S:lke}.

\begin{cor}\label{C:splitsc}
If $S$ is a tree, then
$J_! Sc[S]_* \cong \coprod_{v\in V(S)} \Hom_{\mct_\oper}(T_{\set{v}}, -)$.
\end{cor}

\begin{proof}
The functor $J_!$ is a left adjoint, so preserves coproducts, and we see that
\[ J_! Sc[S]_* = J_! \left( \coprod_{v\in V(S)} \Omega[C_{v}]_* \right) \cong  \coprod_{v\in V(S)} J_!  \Omega[C_{v}]_*. \] By Proposition~\ref{P:lke}, we have
\[ J_! \Omega[C_{v}]_* = \Hom_{\mct_\oper}(J(C_{v}), -) = \Hom_{\mct_\oper}(T_{\set{v}}, -). \]
\end{proof}

A minor variation of the proof of Proposition~\ref{P:lke} gives the following proposition, whose proof we again defer to Section~\ref{S:lke}.

\begin{prop}\label{P:lknerve}
The left Kan extension of $\nerve(T_M)$ along $J$ is $\Hom_{\T_\oper} (T_M, -)$.
\end{prop}

We make use of the following definition in the proof of Proposition~\ref{P:lknerve}.

\begin{defn}\label{D:defnofI}
Define \[ I: \Hom_\coper (R, T_M)  \to \Hom_\oper (J(R), T_M)\] by $I(a)(v) = a(v)$ for each $v\in V(R)$.
\end{defn}

The function $I$ is actually a bijection
since both $a$ and $I(a)$ are completely determined by their values on $V(R)$ and there are no restrictions other than preservation of valence.

At this point we know that the functor $J_!$ takes the diagram of local equivalences
\[ \Omega[S]_* \leftarrow Sc[S]_* \rightarrow \nerve (J(S)) \]
from Proposition~\ref{P:filtration} to local equivalences:
\[ \xymatrix{
J_! \Omega[S]_* \ar@/_3pc/[ddr]^=_{\ref{P:lke}}& J_! Sc[S]_* \ar@{->}[l] \ar@{->}[r] \ar@{->}[d]^=_{\ref{C:splitsc}} & J_! \nerve (J(S)) \ar@/^3pc/[ddl]_=^{\ref{P:lknerve}}\\
& \coprod_{V(S)} \Hom(T_v, -) \ar@{->}[d]^\simeq & \\
& \Hom(J(S), -).
}\]

\begin{prop}\label{P:pullbackofhom} Let $M$ be a set with valences. Then
\[ \nerve (T_M) \cong J^* \Hom_{\mct_\oper}(T_M, -). \]
\end{prop}

\begin{proof}
As mentioned above, $I$ gives a bijection
\[ \nerve(T_M)_S = \Hom_\coper (S, T_M) \overset{I}\cong \Hom_\oper(J(S), T_M) = \Hom_{\mct_\oper}(T_M, J(S)). \]
\end{proof}

\begin{prop}
The adjoint pair
\[ J_! : \dspace_{*,f} \rightleftarrows \sset_*^{\T_\oper} :\! J^*\]
is a Quillen pair.
\end{prop}

\begin{proof}
A map $W_1 \to W_2$ is an (acyclic) fibration in $\sset_*^{\T_\oper}$ if and only if $W_1(T_N) \to W_2(T_N)$ is an (acyclic) fibration in $\sset$ for each $T_N$.
Thus if $W_1 \to W_2$ is an (acyclic) fibration,
\[ W_1(J(S)) = J^*(W_1)_S \to J^*(W_2)_S = W_2(J(S)) \]
is an (acyclic) fibration as well. The fibrations and weak equivalences in $\dspace_{*,f}$ are also defined levelwise, so this implies that $J^*(W_1) \to J^*(W_2)$ is an (acyclic) fibration as well. The result now follows from \cite[8.5.3]{hirschhorn}.
\end{proof}

\begin{lem} \label{L:keyweneed}
For any cofibrant object $X$ in $\dspace_{*,R}$, we have that $L_1 X$ is weakly equivalent to $J^* L_2 J_! X$.
\end{lem}

\begin{proof}
We begin by regarding $X$ as a simplicial object in dendroidal sets and write
\[ X = \hocolim_{\delo}([n] \mapsto \amalg_i \Omega[S_i]_\ast). \]  Applying the localization functor $L_1$, a result which allows us to repeat the localization inside the homotopy colimit \cite[4.1]{simpmon}, and Proposition \ref{P:filtration} to get
\[ \begin{aligned}
L_1 X & = L_1 \hocolim_{\delo}([n] \mapsto \amalg_i \Omega[S_i]_\ast) \\
& \simeq L_1 \hocolim_{\delo} L_1([n] \mapsto \amalg_i \Omega[S_i]_\ast) \\
& \simeq L_1 \hocolim_{\delo}([n] \mapsto \nerve J(\amalg_i S_i)).
\end{aligned} \]

On the other hand, using the fact that $J_!$ commutes with the homotopy colimit and another application of \cite[4.1]{simpmon}, we get
\[ \begin{aligned}
J^*L_2J_! X & = J^* L_2 J_!(\hocolim_{\delo}([n] \mapsto \amalg_i \Omega[S_i]_\ast)) \\
& \simeq J^*L_2( \hocolim_{\delo}J_!([n] \mapsto \amalg_i \Omega[S_i]_\ast)) \\
&\simeq J^* L_2(\hocolim_{\delo} L_2 J_! ([n] \mapsto \amalg_i \Omega[S_i]_\ast)) \\
& \simeq J^* L_2 (\hocolim_{\delo}([n] \mapsto \Hom_{\T_\oper}(\amalg_i \Omega[S_i]_\ast, -))).
\end{aligned} \]

Now, noticing that the functor over which we are taking the homotopy colimit is already local in both cases, we can apply Proposition \ref{P:pullbackofhom} to get the desired result.
\end{proof}

\begin{thm}
The model categories $\el \dspace_{*,f}$ and $\el \sset_*^{\T_\oper}$ are Quillen equivalent.
\end{thm}
\begin{proof}
We first show that $(J_!, J^*)$ is a Quillen pair for the localized model structures. We have localized the category of reduced dendroidal spaces with respect to the set $\class = \set{ Sc[S]_* \hookrightarrow \Omega[S]_*}$ and  $\sset_*^{\T_\oper}$ with respect to the set $\class'= \left\{ \coprod \Hom(T_{m_i}, -) \to \Hom(T_M, -) \right\}$.
Then $J_! \tilde{C} (\class) = J_! \class = \class'$ by Proposition~\ref{P:lke} and Corollary~\ref{C:splitsc}. We apply \cite[3.3.20(1)(a)]{hirschhorn} to see that we still have a Quillen pair after localizing.

We turn to showing that $(J_!, J^*)$ is a Quillen equivalence, using \cite[1.3.16]{hovey}. The first step is to show that $J^*$ reflects weak equivalences between fibrant objects. In both categories under consideration the fibrant objects are precisely the local objects, and local equivalences between local objects are just the usual weak equivalences \cite[3.2.13]{hirschhorn}. Suppose that $A \to B$ in $\el \sset_*^{\T_\oper}$ is a map between local objects such that $J^*A \to J^*B$ is a weak equivalence. We have that $A \to B$ is a weak equivalence if and only if $A(T_M) \to B(T_M)$ is a weak equivalence for all $T_M$. Let $S$ be some tree with $J(S) = T_M$. Then
\[ A(T_M) = AJ(S) = (J^*A)_S \to (J^*B)_S = BJ(S) = B(T_M) \] is a weak equivalence.

We must now show that if $X$ in $\el \dspace_{*,f}$ is cofibrant then $X\to J^* L_2 J_! X$ is a local equivalence, where $L_2$ is the localization functor on $\sset_*^{\T_\oper}$. This is exactly the statement of Lemma~\ref{L:keyweneed}.
\end{proof}

\section{The theory of operads with a group action}

We now extend the above result to the case of operads with a group action.  An action of a group $G$ on an operad $P$ is simply an action of $G$ on $P(n)$ for each $n\geq 0$. We do not insist upon any compatibility with the structure maps of $P$, for doing so would exclude interesting examples such as the circle action on the framed little disks operad, as discussed below in Remark \ref{littledisks}.

\begin{example}
Suppose that $X$ is a $G$-space. Then the usual endomorphism operad $\mathcal E_X$ has an action of $G$. It is defined, for $f\in \mathcal E_X (n) = \map(X^{\times n}, X)$ by
\[ (g\bullet f) (x_1, \dots, x_n) = g\bullet (f(x_1, \dots, x_n)). \]
In fact, if $X$ is a deformation retract of another space $Y$, then $\mathcal E_Y$ inherits an action of $G$.
\end{example}

\begin{rem} \label{littledisks}
Our notion of group action is not a special case of the group operad actions of \cite[\S 2.5]{zhang}. An abelian group $A$ naturally gives rise to three group operads: the first has $G_k=A$ for all $k$, the second has $G_1=A$ and $G_k=\{ e \}$ otherwise, and the third has $G_k=A^k$. However, the equivariance condition of \cite[2.30]{zhang} prevents actions by these group operads in important special cases of interest, such as the circle action on the framed little disks operad. It is worth noting that if we consider the second of these three group operads, then a group operad action is an action in our sense.
\end{rem}

The structure of a group action on an operad can be described using the machinery of algebraic theories.  To see how to understand group actions in this language, we begin by describing the simpler scenario of a group action on a set.

\begin{example} \cite[3.2]{multisort} \label{group}
Consider the category $\mathcal P$ of group actions on sets. Objects of $\mathcal P$ are pairs $(G,X)$ where $G$ is a group and $X$ is a $G$-set.  The morphisms $(G,X) \to (H,Y)$ consist of pairs $(\varphi, f)$ where $\varphi: G\to H$ is a group homomorphism, $f: X\to Y$ is a map of sets, which are subject to the compatibility condition that $f(g\cdot x) = \varphi(g)\cdot f(x)$.
For each
sort $i=1,2$ we have a forgetful map
\[ \varphi_i: \mathcal P \rightarrow \Set \]
and its left adjoint
\[ \lambda_i: \Sets \rightarrow \mathcal P. \]
When $i=1$, we have, for any group $G$ and set $X$,
\[ \varphi_1(G,X) = G \]
(where on the right-hand side $G$ denotes the underlying set of
the group $G$) and for any set $S$
\[ \lambda_1(S)= (F_S, \varnothing) \]
where $F_S$ denotes the free group on the set $S$.

Similarly, when $i=2$, we define
\[ \varphi_2(G,X) = X \]
and
\[ \lambda_2(S) = (e, S) \]
where $e$ denotes the trivial group.

The objects of the theory are representatives of the isomorphism
classes of the $\lambda_1\{1, \ldots, i\} \amalg \lambda_2\{1, \ldots, j\}$
for all choices of $i$ and $j$.
To encode the action of the group on the set, we use the coproduct in $\mathcal P$
\[ (G,X) \amalg (G',X') = (H, (H \times_G X) \amalg (H \times_{G'} X')) \]
where $H= G \ast G'$ is the free product of groups and
\[ H \times_G X = \{(h,x)|h \in H, x \in X\}/\sim  \]
where $(hg,x) \sim (h,gx)$ for any $g \in G$.  We can now take the
opposite of a full subcategory of $\mathcal{P}$ as above to
obtain the corresponding theory.  In particular, the objects of
the theory look like
\[ \lambda_1\{1, \ldots ,i\} \amalg \lambda_2\{1, \ldots, j\}
= (F_i, F_i \times \{1, \ldots ,j\}), \] where $F_i$ denotes the
free group on $i$ generators.
\end{example}

To find the appropriate theory for group actions on operads, we apply the approach of the previous example to the theory of operads rather than to the theory of sets.  We get an $\mathbb N$-sorted theory, but with the degrees shifted by one so that the indices on the operad sorts are consistent with the theory of operads; we could call it an $(\mathbb N \cup \{-1\})$-sorted operad.

Let $\nsgo$ be the category of groups acting on operads.
The objects are pairs $(G, P)$, where $G$ is a group and $P$ is an operad in the category of sets, together with an action of $G$ on $P$ and morphisms which respect the group action.
A morphism $(G,P) \to (G',P')$ consists of a group homomorphism $f:G\to G'$ and an operad map $h: P\to P'$ so that $h(g\bullet p) = f(g) \bullet h(p)$.
Given such a pair, we can define two different kinds of forgetful functors as follows.  First, we have
\[ \varphi_{-1} \colon \nsgo \rightarrow \Sets \] where $\varphi_{-1}(G, P) = G$, the underlying set of the group $G$.  This functor has a left adjoint
\[ \lambda_{-1} \colon \Sets \rightarrow \nsgo \] where for a set $S$, $\lambda_{-1}(S)=(F_S, \star)$ where $F_S$ is the free group on the set $S$ and $\star$ denotes the initial operad such that $\star(1)$ consists of a single point and $\star(n) = \varnothing$ for $n \neq 1$.

For all $n \geq 0$, we also have functors
\[ \varphi_n \colon \nsgo \rightarrow \Sets \] where $\varphi_n(G, P) = P(n)$, and their left adjoints
\[ \lambda_n \colon \Sets \rightarrow \nsgo \] are given by $\lambda_n(S) = (e, P_{S,n})$, where $e$ is the trivial group and $P_{S,n}$ is the free operad on the set $S$ at arity $n$.

The coproduct in $\nsgo$ is given by
\[ (G, P) \amalg (G', P') = (H, (H \times_G P) \amalg (H \times_{G'} P')), \]
where $H=G \ast G'$ and
at level $n$
\[ (H \times_G P)(n) = \{(h,x) \mid h \in H, x \in P(n) \}/((hg, x) \sim (h, gx)). \]
In particular, if we take the coproducts of elements resulting from our left adjoint functors, we get
\[ (e, P) \amalg (e, P') = (e, P \ast P') \] and
\[ (e, P) \amalg (G, \ast) = (G, (G \times P) \amalg (G \times_G \ast)) = (G, G \times P). \]

Thus, we define the objects of the theory $\Tgop$ to consist of finite coproducts of the form
\[ \lambda_{-1}\{1, \ldots ,n_{-1}\} \amalg \lambda_0\{1, \ldots, n_0\} \amalg \cdots \amalg \lambda_k\{1, \ldots ,n_k\} \] where $k \geq -1$ and $n_k\geq 0$.  This object can be more concisely written as the pair
\[ (F_{n_{-1}}, F_{n_{-1}} \times P_{n_0, \ldots ,n_k}) \] where $P_{n_0, \ldots ,n_k}$ denotes the free operad on $n_j$ generators of arity $j$ for each $0 \leq j \leq k$.

\begin{prop}
The category of product-preserving functors $\Tgop \rightarrow \Sets$ is equivalent to the category of operads equipped with group actions.  Similarly, the category of product-preserving functors $\Tgop \rightarrow \SSets$ is equivalent to the category of simplicial operads equipped with simplicial group actions.
\end{prop}

\newcommand{\goperad}{\text{$G$-$\oper$}}

In our other case of interest, where we consider operads equipped with an action of a fixed discrete group $G$, the situation is simpler.
There is a theory $\mathcal T_{\goperad}$ of operads with $G$-action, which can be obtained either by considering free objects in the category of $G$-operads $\goperad$ or by modifying $\mathcal T_\oper$ as in \cite[3.3.5.h]{borceux2} for $G$-sets.

\begin{prop}
The category of product-preserving functors $\mathcal T_\goperad \rightarrow \Sets$ is equivalent to the category of operads equipped with a $G$-action.  Similarly, the category of product-preserving functors $\mathcal T_\goperad \rightarrow \SSets$ is equivalent to the category of simplicial operads equipped with a $G$-action.
\end{prop}

\section{Comparison with a simpler model}

Just as the theory $\T_\oper$ could be ``replaced" in some sense by the category $\omop$, our goal in this section is to find a simpler category which can take the place of $\Tgop$.

We begin by recalling the definition of Bousfield-Segal groups from \cite{addinverse}. There, we considered \emph{Segal premonoids}, or functors $X \colon \delo \rightarrow \sset$ such that $X_0 = \Delta[0]$.  There are model structures $\SSp_{*,f}$ and $\SSp_{*,R}$ given by reducing the projective and Reedy model structures on the category of simplicial spaces.

In ${\bf \Delta}$, consider the
maps $\gamma^k:[1] \rightarrow [n]$ given by $0 \mapsto 0$ and $1
\mapsto k+1$ for all $0 \leq k < n$.  Restricting to Segal
premonoids, we can define the \emph{Bousfield-Segal map} $\psi_n:X_n
\rightarrow (X_1)^n$ induced by these maps. The models for
simplicial groups in this sense are the Segal premonoids for which the Bousfield-Segal maps are weak
equivalences of simplicial sets for all $n \geq 2$.  We call such
simplicial spaces \emph{Bousfield-Segal groups}.

To motivate this definition, we explain the situation briefly in the case when the maps $\psi_n$ are \emph{isomorphisms}, rather than weak equivalences.
Considering dimension two, the inclusion
\[ \Delta[1] \vee \Delta[1] = ( \, \overset{2}\bullet \leftarrow \overset0\bullet \rightarrow \overset1\bullet) \hookrightarrow \Delta[2] \]
incudes an isomorphism\footnote{To generalize to the groupoid case where $X_0 \neq *$, one should use the identification \[\map(\, \overset{2}\bullet \leftarrow \overset0\bullet \rightarrow \overset1\bullet, X) \cong X_1 \,_{d_1}\!\times_{d_1} X_1.\]}
\[ (a,b) \in X_1 \times X_1 = \map(\, \overset{2}\bullet \leftarrow \overset0\bullet \rightarrow \overset1\bullet, X) \overset\cong\leftarrow \map(\Delta[2], X) = X_2 \ni x \]
and we end up with a new element $[a,b] = d_0(x)$ as in Figure~\ref{F:boustwo}.

\begin{figure}
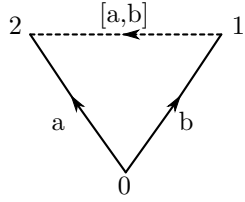
\caption{Binary operation defined by $d_0$}\label{F:boustwo}
\end{figure}
In the case of strict Segal monoids, one turns to dimension three to establish associativity. In our case, the dimension three picture (Figure~\ref{F:bousthree}) gives us a different compatibility relation.
\begin{figure}
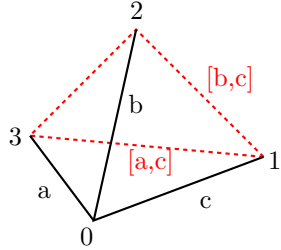
\caption{Labellings of two red edges come from $d_2$ and $d_0$}\label{F:bousthree} \end{figure}
Namely, the simplicial identity $d_0d_1 = d_0d_0$ tells us that $[a,b] = [[a,c],[b,c]]$, as one can see graphically from Figures~\ref{F:bousthree} and \ref{F:bousrelation}.
\begin{figure}
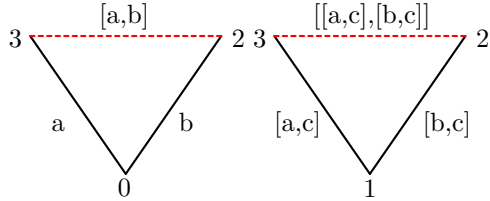
\caption{The edge $\overset2\bullet \to \overset3\bullet$ obtained in two ways}\label{F:bousrelation}
\end{figure}
In the case when $X_0 = *$, let $e=s_0(*)$ and we have the following set of relations, the last two of which the reader may check using the simplicial identities:
\begin{itemize}
\item $[a,b] = [[a,c],[b,c]]$,
\item $[a,e]=a$ (using $s_0(a) = (a,e)$), and
\item $[a,a]=e$ (using $s_1(a) = (a,a)$).
\end{itemize}
A binary operation satisfying these conditions is the same as a group with $[a,b]=ab^\inv$, as observed in \cite{mhall}. If we drop the condition that $X_0 = *$, but still require $X_0$ to be discrete, we exactly recover groupoids.

Any simplicial set $K$ can be regarded as a simplicial space $\delo \rightarrow \sset$ by a constant diagram, with simplicial set $K$ at level $n$ for all $n \geq 0$; we still denote this simplicial space by $K$.
However, we can also regard it as constant in the opposite simplicial direction by taking the set $K_n$ as a constant simplicial set at level $n$.  We denote by $K^t$ this ``transposed" simplicial space, which we use in the following definitions.  To give a localized model structure on reduced simplicial spaces, we define for each $k \geq 2$
the simplicial space
\[ Star(k)^t_* = \bigcup_{i=1}^{k-1} \gamma^i \Delta[1]^t_* \subseteq
\Delta[k]^t_*. \]  Then, as in the previous situations, define the
map
\[ \psi_* \colon  \coprod_{k \geq 1} Star(k)^t_* \rightarrow \Delta [k]^t_*.
\]

\begin{prop} \cite[6.1]{addinverse}
Localizing the model category structure $\SSp_{*,f}$ with respect
to the map $\psi_*$ results in a model category structure
$\mathcal L_B \SSp_{*,f}$ whose fibrant objects are reduced
Bousfield-Segal groups.  There is also an analogous model
structure $\mathcal L_B \SSp_{*,R}$.
\end{prop}

\newcommand{\delom}{\mathbf{\Delta \circlearrowright \Omega}}
\newcommand{\bi}[2]{\ensuremath{  \left[ #1 \circlearrowright #2       \right]      }}
\newcommand{\opact}{{\mathcal{O}p{A}ct}}
\newcommand{\pullback}[2]{\,_{#1}\!\times_{#2}}

We now recall the category $\delom$ from \cite{bergnerhackney} which models actions of categories on colored operads. Briefly, an \emph{action} of a category $\mcc$ on a colored operad $P$ is an action of $\mcc$ on the set $\mor(P)$ of morphisms of $P$ satisfying three additional compatibility conditions. Category actions on sets are defined in the same way as groupoid actions: we have a moment map $\mu: \mor P \to \ob \mcc$ and an action map
\[ \bullet: \mor(\mcc) \pullback{s}{\mu} \mor(P) \to \mor(P) \]
which satisfy
\begin{equation*}
\mu(f\bullet g) = t(f), \qquad \bullet(\id,\bullet) = \bullet(\circ,\id), \text{ and} \qquad \id_{\mu(g)} \bullet g = g.
\end{equation*}
The additional axioms are $\mu(\gamma(g; g_1, \dots, g_k)) = \mu(g)$ where $\gamma$ is the operadic composition, $s(f\bullet g) = s(g)$ as ordered lists of colors of $P$, and $\sigma^* (f\bullet g) = f\bullet (\sigma^* g)$  for $\sigma$ a permutation.

The category $\delom$ is a full subcategory of $\mathcal{COA}ct$, the category whose objects are categories acting on colored operads. Objects in $\delom$ are of the form $\bi{n}{R}$ where $n\geq 0$ and $R$ is a tree or the empty set. Such an object is characterized by the following: if the category $\mcc$ acts on a colored operad $P$, then the set
\[ \Hom_{{\mathcal{COA}ct}}(\bi{n}{R}, \mcc
\circlearrowright P)\]
is given by maps
\begin{align*} \alpha: [n] &\to \mcc \\
\beta: R &\to P
\end{align*}
such that $\alpha(0) = \mu (\beta(r))$, where $r$ denotes the root of $R$.
We write $\SSets^{\delom^{op}}$ for the category of functors $\delom^{op} \rightarrow \SSets$, which admits a projective model structure $\SSets^{\delom^{op}}_f$ and a generalized Reedy model structure $\SSets^{\delom^{op}}_R$ per \cite{bergnerhackney}.

The appropriate set of maps for localizing these model structures can be obtained by combining the Segal maps on $\om$ with the Bousfield-Segal maps on $\del$ in the following way. We have maps
$\gamma^k: \bi{1}{\varnothing} \to \bi{n}{R}$ defined by sending $0\mapsto 0$ and $1\mapsto k+1$, for $0\leq k < n$. As in Definition~\ref{D:segalcore}, we have, for each vertex $v$ of $R$ a map
\[ \zeta^v: \bi{0}{C_v} \to \bi{n}{R}. \] If $e$ is the output edge of $v$, then this map sends the vertex of $C_v$ to $v$ and $0$ to $\mu(e)$.

For brevity, in the following definition we denote by
\[ rep\bi{n}{R} = \Hom_\delom(-, \bi{n}{R})\] the representable functor.
\begin{defn}
A \emph{Segal core} of $rep\bi{n}{R} = \Hom_\delom(-, \bi{n}{R})$ is
\[ Sc\bi{n}{R} = \left(\bigcup_{k=0}^{n-1} \gamma^k rep\bi{1}{\varnothing} \right) \cup \left(\bigcup_{v\in  V(R)} \zeta^v rep\bi{0}{C_v} \right) \subseteq rep\bi{n}{R} \]
\end{defn}

Denote by $\SSets_*^{\delom^{op}}$ the category whose objects are those functors $X: \delom^{op} \to \SSets$ such that $X(\bi{0}{\varnothing}) = \Delta[0]$.
We localize this category by the reductions of the above core inclusions and denote the resulting localized model structure by $\LSSets^{\delom^{op}}$.
There is a functor $\delom \to \Tgop^{op}$ given by $\bi{n}{R} \mapsto (F_n, F_n \times J(R))$ where $F_n$ is the free group on $n$ generators and $J(R)$ is the operad generated by the vertices of $R$.  This induces a functor $\LSSets^{\Tgop}_* \rightarrow \LSSets^{\delom^{op}}_*$, which has a left adjoint given by left Kan extension.

The proof of the following result can be obtained by combining the arguments for Bousfield-Segal groups and the one for Segal operads as given earlier in this paper.

\begin{thm}
The adjoint pair
\[ \LSSets^{\Tgop}_{*,f} \leftrightarrows \LSSets^{\delom^{op}}_{*,f} \]
is a Quillen equivalence of model categories.
\end{thm}

Notice that if we restrict to the case where $R=\eta$, we recover the comparison between simplicial groups and Bousfield-Segal groups.  If we restrict to the case where $n=0$, we recover the comparison between simplicial operads and Segal operads.

In the case of operads equipped with an action by a discrete group $G$, we form the product category $G\times \om$, where $G$ is regarded as a category with objects given by the elements of $G$ and no non-identity maps.  These objects form a generalized Reedy category since $G$ is a generalized Reedy category and the class of generalized Reedy categories is closed under finite products.  Localizing with respect to the set of maps \eqref{E:class}, the fibrant objects of $\LSSets_*^{(G\times \om)^{op}}$ should be regarded as Segal operads equipped with an action of $G$. Moreover, if we let $J_G(R)$ be the free $G$-operad on the vertices of $R$, there is a functor $G\times \om \to \mathcal T_{\goperad}^{op}$ which sends an object $*\times R$ to $J_G(R)$, and a morphism $(g,f): *\times R \to * \times S$ to $g\circ J_G(f)$.

\begin{thm}
The adjoint pair
\[ \LSSets^{\mathcal T_{\goperad}}_{*,f} \leftrightarrows \LSSets^{(G\times \om)^{op}}_{*,f} \]
is a Quillen equivalence of model categories.
\end{thm}

The proof follows just as in the operad case.

\section{Technical results}

\subsection{Proof of Proposition~\ref{P:filtration}}\label{S:filtrprop}

\theoremstyle{plain}
\newtheorem*{filtrprop}{Proposition \ref{P:filtration}}

\begin{filtrprop} Let $P$ be the free operad on the generating set $M=\set{x_1^{j_1}, \dots, x_m^{j_m}}$, where $x^p$ is in arity $p$. If $S$ is any tree whose list of sub-corollas is $C_{j_1}, \dots, C_{j_m}$, then
\[ \Omega[S]_* \simeq \nerve(P) \]
in the localized model structure $\el\dspace_{*,R}$.
\end{filtrprop}

Before we prove this proposition, we give a series of lemmas which break down the major parts.

Write $\partial^{ext} \Omega[R]_* \ci \Omega[R]_*$ for the images of the external faces $S \ci R$ where $|S| = |R| -1$. Writing $\inclo: S \to R$ for the inclusion of the subtree, we have
\[ \partial^{ext} \Omega[R]_* = \bigcup \inclo (\Omega[S]_*). \]

\begin{lem}\label{L:acyc}
The inclusion $\partial^{ext} \Omega[S]_* \to \Omega[S]_*$ is an acyclic cofibration in the localized model structure $\el\dspace_{*,R}$.
\end{lem}
\begin{proof}

By Lemmas~\ref{L:subthings} and \ref{L:rrnormal}, $Sc[S]_* \to \Omega[S]_*$ and $\partial^{ext} \Omega[S]_* \to \Omega[S]_*$ are cofibrations. Since $Sc[S]_* \to \Omega[S]_*$ is an acyclic cofibration in the localized model structure, it suffices to show that $Sc[S]_* \to \partial^{ext} \Omega[S]_*$ is a weak equivalence.

We can rewrite the usual definition
\[ Sc[S]_* = \bigcup_{v\in V} \inclo\Omega[C_{n(v)}]_* \ci \Omega[S]_* \]
as
\[ Sc[S]_* = \bigcup_R \inclo Sc[R]_* \ci \Omega[S]_* \]
where $R$ ranges over all proper subtrees of $S$, which is equal to the union over all proper subtrees with $|R| = |S|-1$. We claim that
\[ Sc[S]_* = \bigcup_R \inclo Sc[R]_* \to \bigcup_R \inclo \Omega[R]_* \] is an acyclic cofibration.  To establish this claim, we note that the diagram
\[ \xymatrix{ \coprod_R Sc[R]_* \ar@{->}[r] \ar@{->}[d] & \bigcup_R \inclo Sc[R]_* \ar@{->}[d]\\
\coprod_R \Omega[R]_* \ar@{->}[r] & \bigcup_R \inclo \Omega[R]_*,
}\]
is a pushout, and the left hand map is an acyclic cofibration so that the right hand map $Sc[S]_* \to \partial^{ext} \Omega[S]_*$ is also an acyclic cofibration.
\end{proof}

If $A$ is a free monoid on generators $a_1, \dots, a_n$, then we can consider $\nerve(A)$ as the simplicial set generated by $r$ simplices of the form
\[ * \overset{a_{j_1}}\longrightarrow * \overset{a_{j_2}}\longrightarrow  \dots \overset{a_{j_r}}\longrightarrow  * \] where all of the arrows are labeled by generators of $A$. Since we can obtain all other simplices of $\nerve(A)$ by applying face and degeneracy maps to simplices of the above form, we call such simplices \emph{primitive}.

We now examine the corresponding situation for the ``primitive dendrices'' of $\nerve(P)$. Primitive simplices are labellings of linear trees by the generating set,
\[ \xymatrix{ \ar@{-}[r] & \overset{a_{j_1}}{\bullet} \ar@{-}[r] & \overset{a_{j_2}}\bullet \ar@{-}[r] & \cdots \ar@{-}[r] & \overset{a_{j_r}}{\bullet} \ar@{-}[r] & } \]
so we expect primitive dendrices to be labels of the vertices of an arbitrary tree. However, note that the existence of tree automorphisms causes some minor complications.

Let $ns\coper$ denote the category of non-symmetric colored operads and consider the adjoint pair
\[ \Sigma \colon ns\coper \rightleftarrows \coper \, : \! U, \]
where $\Sigma$ is the symmetrization functor.
To avoid confusion, we will use the following notation for the remainder of this subsection. If $S$ is a tree and $p$ is a planar structure on $S$, then the planar tree $(S,p)$ generates a \emph{non-symmetric} colored operad which we denote $\om_p(S,p)$. For emphasis, we write $\om(S)$ for $\Sigma\om_p(S,p)$; we have been writing $S$ for this colored operad elsewhere.
Notice that if $(S,p)$ is a planar tree then
\[ \Hom_\coper (\om(S), P)  = \Hom_{ns\coper} (\om_p(S,p), U(P)). \]

\begin{defn}
For each tree $S$, the set \[ \wprim_S \subseteq \nerve(P)_S \] is the subset of all $\alpha: \om(S)\to P$ with the property that there exists a planar representative $(S,p)$ such that the map of non-symmetric operads
\[ \alpha: \om_p(S,p) \to U(P) \]
restricts to a map $\check\alpha: V(S) \to M$ (here $M=\set{x_1^{j_1}, \dots, x_m^{j_m}}$ is the set of generators of $P$).  The group $\Aut(S)$ acts on $\wprim_S$, and we let
\[ \prim_S = \wprim_S/{\Aut(S)}\]
be the coinvariants, which we call the set of \emph{primitive elements} of $S$.
\end{defn}

\begin{rem}
The primitive elements of $S$ are in bijection with the labels of the vertices of $S$ by elements of $M$, subject to the condition that valences match.
\end{rem}

\begin{rem}\label{R:sortingout}
Suppose $\alpha: \om_p(S,p) \to U(P)$ has the property that for each $v\in V(S)$, $\alpha(v) = \tau x_i$ for some $\tau \in \Sigma_{|v|}$. Then there is a planar structure $p'$ on $S$ so that the map $\alpha: \om_p(S,p') \to U(P)$ takes $v$ to $x_i$.
\end{rem}

\begin{lem}\label{L:genned} Consider the inclusion
\[ \coprod_{S} \wprim_S \to \coprod_{S} \nerve(P)_S. \]
The sub-dendroidal set generated by $\coprod_S \wprim_S$ is $\nerve(P)$.
\end{lem}
\begin{proof}
Suppose that we have $\beta \in \nerve(P)_S$ and fix a planar representative of $S$. We first show that $\beta$ can be `spread apart'. Suppose that $\beta(v) = f\circ_{i+1} g$. Consider the tree $S'$ which looks like $S$ away from $v$: it has $E(S') = E(S) \sqcup \set{e}$ and $V(S') = (V(S) \setminus \set{v}) \sqcup \set{v_1, v_2}$, and the changes are pictured in Figure~\ref{F:dele}. We then have $\partial^e: S \to S'$. Choosing compatible planar structures on $S$ and $S'$, we can define $\beta': \om_p(S',p') \to P$ on vertices by insisting that $\beta'(v_2) = g$, $\beta'(v_1) = f$, and $\beta'(w) = \beta(w)$ for all other vertices of $S'$. Then $\beta = \beta' \circ \partial^e = \partial_e (\beta')$.
\begin{figure}
\includegraphics[width=0.9\textwidth]{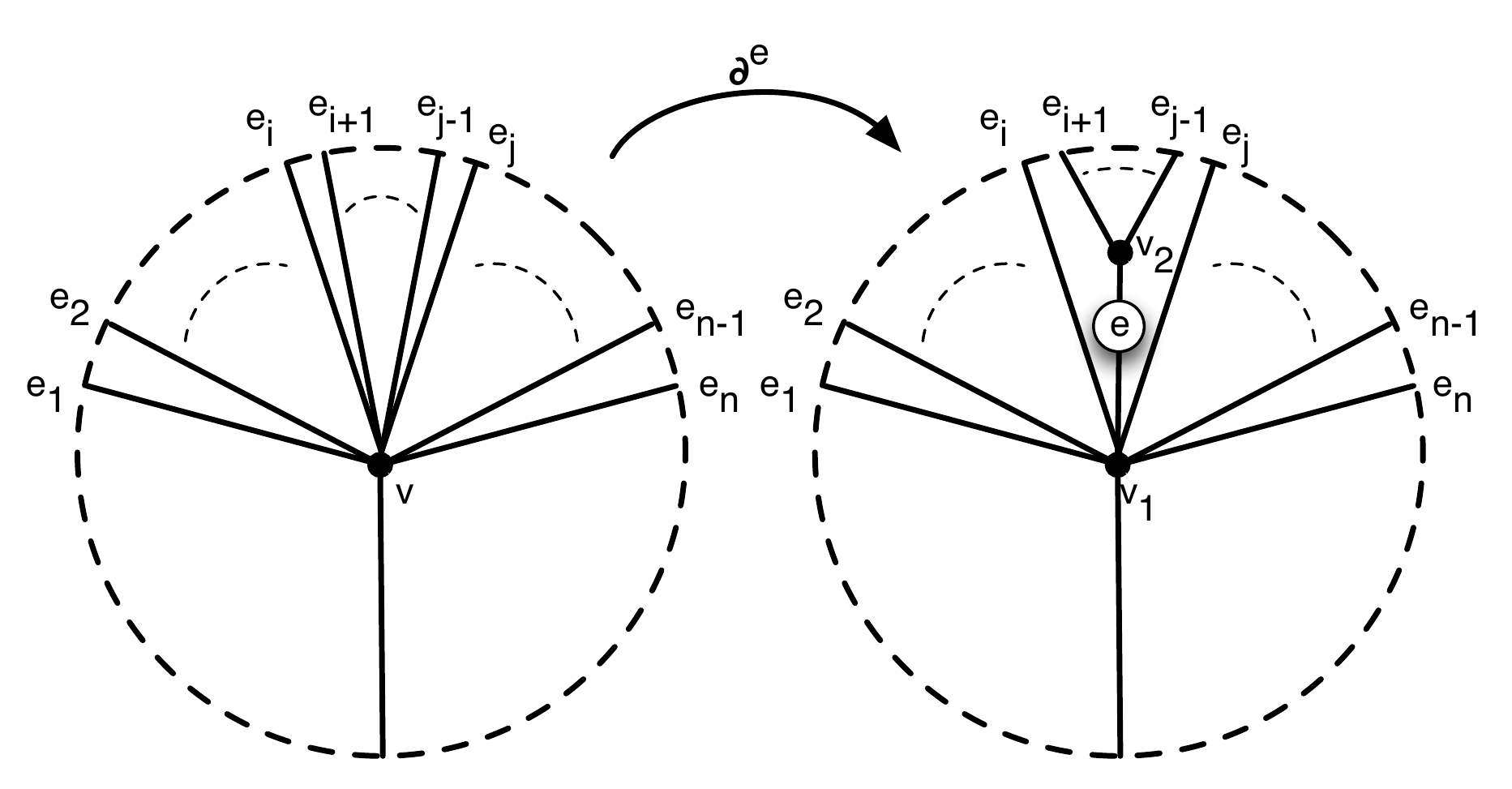}
\caption{an inner coface map}\label{F:dele}
\end{figure}
Iterating, we may assume that for a choice of planar structure $p$, $\beta$ has the property that
 $\beta(v) = \id$ or  $\beta(v) = \tau x_k$ for some $\tau \in \Sigma_{|x_k|}$.

If $\beta(v) = \id$, then $\beta=\sigma_v \beta'$, where $\sigma^v$ is pictured in Figure~\ref{F:upv}.
\begin{figure}
\includegraphics[width=0.5\textwidth]{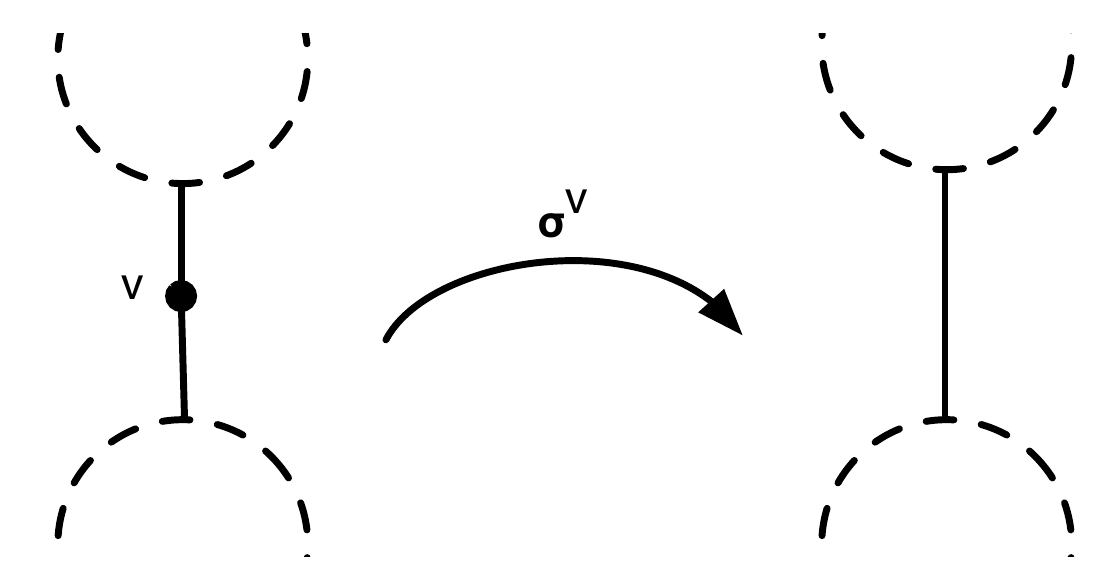}
\caption{a codegeneracy map}\label{F:upv}
\end{figure}
The source of the map $\beta'$ looks like $S$ except that the tree has $v$ omitted; iterating we may assume that $\beta(v) = \tau x_k$ for each vertex $v$.

At this point, we have a map of planar operads $\beta: \om_p(S,p) \to U(P)$ with
$\beta(v) = \tau_v x_{i_v}$ for $v\in V(S)$. The result now follows from Remark~\ref{R:sortingout}.
\end{proof}

Filter $\nerve(P)$ as
\[ \Psi^1 \ci \Psi^2 \ci \dots \ci \nerve(P) \]
by letting $\Psi^n$ be the sub-dendroidal set generated by elements of $\wprim_S$ for $|S|\leq n$.

\begin{lem}
This filtration is exhaustive, i.e.   $\nerve(P) = \bigcup \Psi^n$.
\end{lem}
\begin{proof}
We need to show that if $f\in \nerve(P)_S$, then there is an $n$ so that $f\in \Psi^n_S$. By Lemma~\ref{L:genned}, $f=\gamma^* g$ where $\gamma$ is a map in $\Omega$ and $g\in \prim_R$. Then $f\in \Psi^{|R|}_S$.
\end{proof}

\begin{lem}\label{L:subtreefiltration}
Suppose $\alpha \in \wprim_S$ with $|S|=n$ and that $\inclo: R\hookrightarrow S$ is a proper subtree.
Then $\alpha \inclo \in \wprim_R$ has the property that the image of
\[ \alpha \inclo: \Omega[R]_* \to \nerve(P) \] is in $\Psi^{n-1}$.
\end{lem}
\begin{proof}
Generally, for a tree $S$  and $\alpha \in \wprim_S$, we have that the image of \[ \alpha: \Omega[S]_* \to \nerve(P) \] is in the sub-dendroidal set generated by $\wprim_S$, which is contained in $\Psi^{|S|}$.
Thus the proof is just the observation that since $R$ is a proper subtree, $|R| < |S| = n$.
\end{proof}

Suppose $\alpha \in \wprim_S$ where $|S|=n$ and consider the maps $\inclo^R: R\hookrightarrow S$ for every subtree of $S$ with $|R| = n-1$.
Notice that the maps \[\alpha \inclo^R: \Omega[R]_* \to \Psi^{n-1} \ci \nerve(P)\] from Lemma~\ref{L:subtreefiltration} together give a map
\[ \widetilde{\alpha}: \partial^{ext} \Omega[S]_* \to \Psi^{n-1}. \]

\begin{proof}[Proof of Proposition~\ref{P:filtration}] Recall that we are regarding $\om$ as a skeletal category, so $\ob \om$ is a set.
For each $\bar{\alpha} \in \prim_S$ with $S\in \ob \om$, choose a representative $\alpha \in \wprim_S$.
Observe that
\[ \xymatrix{ \coprod_{|S|=n} \coprod_{\bar\alpha \in \prim_S} \partial^{ext} \Omega[S]_* \ar@{->}[r]^-{\coprod \tilde{\alpha}} \ar@{->}[d] & \Psi^{n-1}\ar@{->}[d] \\
\coprod_{|S|=n}  \coprod_{\bar\alpha \in \prim_S}  \Omega[S]_* \ar@{->}[r]^-{\coprod \alpha} & \Psi^n
 } \]
is a pushout. The left-hand map is an acyclic cofibration by Lemma~\ref{L:acyc}, so $\Psi^{n-1} \to \Psi^n$ is an acyclic cofibration by \cite[7.2.12(2)]{hirschhorn}.

Finally, $\prim^1 = M$ and the only trees with one vertex are the corollas. So
\[ \Psi^1 = \coprod_n \coprod_{\prim_{C_n}} \Omega[C_n]_* \cong Sc[S]_* \] by our assumption on $S$. Thus
\[ \Omega[S]_* \simeq Sc[S]_* \cong \Psi^1 \simeq \hocolim \Psi^i  \simeq \colim \Psi^i = \nerve(P). \]

\end{proof}

\subsection{Proofs of Propositions~\ref{P:lke} and \ref{P:lknerve}}\label{S:lke}

\theoremstyle{plain}
\newtheorem*{lkeprop}{Proposition \ref{P:lke}}

Recall that we defined $J^{op}: \om \to \T_{\oper}^{op}$ in \S\ref{SS:comparison} by sending a tree $S$ to the free operad on $V(S)$. The functor  $J: \om^{op} \to \T_\oper$ is its opposite.

\begin{lkeprop} The left Kan extension of $\Omega[S]_*$ along $J$ is  $\Hom_{\T_\oper} (J(S), -)$.
\end{lkeprop}

The strategy is to  evaluate $J_! \Omega[S]_*$ at some free operad $T_N$ and show that this is equal to $\Hom_{\T_\oper}(J(S), T_N)$.  The usual formula for left Kan extension as a colimit (see, for example, \cite[X.3, 1]{maclane}) gives
\[ J_! (\Omega[S]_* )(T_N) = \underset{g\in J / T_N}{\colim} (\Omega[S]_*)_R  =
\underset{g\in J / T_N}{\colim} \Omega[S]_{*,R,g}, \]
where the objects of $J/ T_N$ are maps  $g:J(R) \to T_N$ in $\T_\oper$.
From now on, we work in $\oper$ rather than $\T_\oper$. In particular, we want to show \[\Hom_{\T_\oper} (J(S), T_N) = \Hom_\oper(T_N,J(S)).\]

We have the following standard fact (see \cite[V.2]{maclane} and \cite[2.4.6b]{borceux}), which we use to rewrite the above colimit.

\begin{prop}\label{P:quotientcolim}
Let $F: \mcc  \to \sset$
be a functor. Then
\[ \colim_\mcc F = \left( \coprod_{c\in \ob \mcc} F(c) \right) / \sim \]
where the equivalence relation is generated by applying maps in $\mcc$. Specifically, for each $\alpha: c\to c'$ in $\mor \mcc$ we have $F(\alpha): F(c) \to F(c')$, and we declare that $x\sim F(\alpha)(x)$.
\end{prop}

With the equivalence relation from this proposition, we write
\[ \underset{g\in J / T_N}{\colim} \Omega[S]_{*,R,g} = \left( \coprod_{g\in J/T_N} \Omega[S]_{*,R,g} \right) / \sim. \]
where $\sim$ is the equivalence relation from Proposition~\ref{P:quotientcolim}. We want to define a map from this colimit to $\Hom_{\oper} (T_N,J(S))$, and we will first define a map
\[ \widetilde{\newPsi}: \coprod_{g\in J/T_N} \Omega[S]_{*,R,g} \to \Hom_{\oper} (T_N,J(S)) \]
on components. Fix $g:T_N \to J(R)$ a map of operads, and let us define $\widetilde{\newPsi}$ on the  $g$-component $\Omega[S]_{*,R,g}$. For a map $a:R\to S$ of colored operads, write $\bar a$ for the equivalence class of $a$ in $(\Omega[S]_*)_R$. We would like to show that the assignment
\[ \widetilde{\newPsi}(\bar a) = J(a)\circ g \in \Hom_\oper(T_N, J(S)) \]
is well-defined.

If $R$ is nonlinear, then, $\Omega[S]_{*,R}$ is just $\Omega[S]_R = \Hom_\oper (R, S)$, so we need only check the case where $R=[n]$ is a linear tree. Note that $\Hom_\oper(T_N, J([n])) \ci  \ob(J/T_N)$ is empty if $N$ has elements of valence other than $1$, so we may assume that all elements of $N$ have valence $1$. Let $g: T_N \to J([n])$ be a map of operads.
We would like to send the equivalence class $\bar{a} \in \Omega[S]_{*,[n]}$ to $J(a)g$.
To see that $\bar{a} \mapsto J(a)g$ is well-defined, consider the following diagram where the square is a pushout.
\[ \xymatrix{
\coprod_{E(S)} \Omega[\eta]_{[n]} \ar@{->}[r] \ar@{->}[d] & \Omega[\eta]_{[n]} \ar@{->}[d] \ar@/^/[ddr]^\ast \\
\Omega[S]_{[n]} \ar@{->}[r] \ar@/_/[drr]_{a\mapsto J(a)g}& \Omega[S]_{*,[n]} \ar@{-->}[dr] \\
&& \Hom_\oper(T_N, J(S))
}\]
The map $T_N \to J(S)$ which sends each generator of $N$ to the unit of $J(S)$ defines the map $*=\Omega[\eta]_{[n]} \to \Hom_\oper(T_N, J(S))$.
The outer diagram commutes, and we get the induced map \[ \Omega[S]_{*,[n]} \to \Hom_\oper(T_N, J(S))\] since $\Omega[S]_{*,[n]}$ is a pushout.

\begin{lem}\label{L:41colim} The map
\[ \widetilde{\newPsi}: \coprod_{J/ N} \Omega[S]_{*,R,g} \to \Hom_\oper(T_N, J(S)) \] induces a map
\[ \newPsi: \underset{f\in J / T_N}{\colim} \Omega[S]_{*,R,g} \to \Hom_\oper(T_N, J(S)). \]
\end{lem}

\begin{proof}
Suppose that we have a map in $J/ T_N$, which is given by a commutative diagram of  maps in $\oper$
\begin{equation}\label{E:diagram1}
\xymatrix{ T_N \ar@{->}[r]^g \ar@{->}[dr]_{g'} & J(R) \ar@{->}[d]^{J(b)} \\
& J(R') }
\end{equation}
where $b: R\to R'$ is a map in $\coper$. We use the shorthand $b: g' \to g$ for the above commutative diagram.
Take $\bar{a}\in \Omega[S]_{*,R,g}$ and $\bar{a}'\in \Omega[S]_{*,R',g'}$
such that $b: g' \to g$ takes $\bar{a}'$ to $\bar{a}$. In other words, there are representatives $a: R\to S$ and $a': R' \to S$ so that the diagram
\begin{equation}\label{E:diagram2} \xymatrix{
R \ar@{->}[d]^b \ar@{->}[dr]^{a} \\
R' \ar@{->}[r]^{a'} & S
}\end{equation}
commutes.
Since the images of $\bar{a}$ and $\bar{a}'$ in the colimit are equal, we must show that $\widetilde{\newPsi}(\bar{a}) = \widetilde{\newPsi}(\bar{a}')$.

Applying $J$ to \eqref{E:diagram2} and combining the result with \eqref{E:diagram1} gives a commutative diagram
\[ \xymatrix{ T_N \ar@{->}[r]^g \ar@{->}[dr]_{g'} & J(R) \ar@{->}[d]^{J(b)} \ar@{->}[d]^{J(b)} \ar@{->}[dr]^{J(a)} \\
 & J(R') \ar@{->}[r]^{J(a')} & J(S) }\]
which shows that
\[ \widetilde{\newPsi}(\bar{a}) = J(a)g = J(a')g' = \widetilde{\newPsi}(\bar{a}').\]
Thus we get the map $\newPsi$
\[ \xymatrix{
\coprod_{J/ N} \Omega[S]_{*,R} \ar@{->}[rr]^{{\widetilde{\newPsi}}} \ar@{->}[dr] & &\Hom(T_N, T_M)\\
& \colim_{J/ N} \Omega[S]_{*,R} \ar@{-->}[ur]_{\newPsi}
}\]
as desired.
\end{proof}

Now that we have defined the map $\newPsi$ at $T_N$, we must show that $\newPsi$ is an isomorphism.

\begin{lem}
The map $\newPsi$ is bijective.
\end{lem}

\begin{proof}
We begin by showing that $\newPsi$ is injective. Suppose that we have (colored) operad maps
\begin{align*}
a&: R\to S  & a'&: R'\to S\\
g&: T_N\to J(R) & g'&: T_N\to J(R')
\end{align*}
with the property that $J(a)g = J(a')g'$. Considering $\bar a$ in the $g$ component and $\bar a'$ in the $g'$ component, this property is just the condition that $\widetilde\newPsi (\bar a) = \widetilde\newPsi(\bar a')$.  Write $h=J(a)g$.
The following diagram commutes
\[ \xymatrix{
& J(R) \ar@{->}[dr]^{J(a)} \ar@{->}[d]_{J(a)}& \\
T_N \ar@{->}[ur]^g \ar@{->}[r]^h \ar@{->}[dr]_{g'} & J(S) \ar@{->}[r]^{\id} &J(S) \\
& J(R') \ar@{->}[u]^{J(a')} \ar@{->}[ur]_{J(a')}&
}\]
which gives a zig-zag
\[ g \overset{a}{\longrightarrow} h \overset{a'}{\longleftarrow} g' \]
in $(J/ T_N)^{op}$. Thus $a \sim \id \sim a'$ in the equivalence relation defining $\colim \Omega[S]_{*,R}$, so $\newPsi$ is injective.

We now turn to surjectivity.
Suppose that $h \in \Hom_\oper(T_N, J(S))$.
Consider $\id: S \to S$ as
\[ \id\in \Omega[S]_{*,S,h} \subseteq \coprod_{J/T_N} \Omega[S]_{*,R,g}, \]
so $\newPsi[\id] = \widetilde\newPsi(\id) = J(\id) g = g$.
\end{proof}

\theoremstyle{plain}
\newtheorem*{lknerve}{Proposition \ref{P:lknerve}}

\begin{lknerve}
The left Kan extension of $\nerve(T_M)$ along $J$ is  $\Hom_{\T_\oper} (T_M, -)$.
\end{lknerve}

\begin{proof}[Sketch of Proof]
One must make several modifications to the proof of Proposition~\ref{P:lke}. The first change is to replace $\Omega[S]_*$ by $\nerve(T_M)$ throughout.

It is slightly easier to define a map
\[ \widetilde{\newPsi}: \coprod_{g\in J/ T_N} (\nerve(T_M)_R )_g \to \Hom_\oper(T_N, T_M) \]
in this setting.
Recall, from Definition~\ref{D:defnofI}, the bijection \[ I: \Hom_\coper (R, T_M)  \to \Hom_\oper (J(R), T_M)\] with $I(a)(v) = a(v)$ for each $v\in V(R)$.
We define $\widetilde{\newPsi}$ on the $g$-component by  \[ \widetilde{\newPsi}(a) = I(a) g. \]

Elsewhere, in the proof of Proposition~\ref{P:lke}, we use $J(a)$ and $J(S)$ where we previously used $I(a)$, $T_M$, and we no longer need to consider representatives.

To show surjectivity, suppose that $g\in \Hom_\oper(T_N,T_M)$ and let $S$ be a tree together with a fixed isomorphism $\tilde h: V(S){\to} N$. Define $h: S\to T_N$ on vertices by $h(v) = \tilde h(v) \in M$. Since $T_N$ is a single-colored operad, $h$ has been completely determined. We now have operad maps
\begin{align*} J(S) &\overset{I(h)}{\longrightarrow} T_N \overset{g}{\to} T_M \\
T_N &\overset{I(h)^\inv}{\longrightarrow} J(S) \end{align*}
and we set
\begin{align*}
a &= I^\inv ( g\circ I(h)) &  a: S &\to T_M \\
f&= I(h)^\inv.
\end{align*}
Consider the element $a\in \nerve(T_M)(S)_f$ and calculate
\[ \widetilde{\newPsi}(a) = I \big(  I^\inv (g\circ I(h)) \big) \circ I(h)^\inv = g I(h) I(h)^\inv = g. \] Since $g$ was arbitrary we see that $\widetilde{\newPsi}$ is surjective, hence $\newPsi$ is as well.
\end{proof}

\end{document}

%% file: bousfield2simplexmesss.pdf_tex
\begingroup%
  \makeatletter%
  \providecommand\color[2][]{%
    \errmessage{(Inkscape) Color is used for the text in Inkscape, but the package 'color.sty' is not loaded}%
    \renewcommand\color[2][]{}%
  }%
  \providecommand\transparent[1]{%
    \errmessage{(Inkscape) Transparency is used (non-zero) for the text in Inkscape, but the package 'transparent.sty' is not loaded}%
    \renewcommand\transparent[1]{}%
  }%
  \providecommand\rotatebox[2]{#2}%
  \ifx\svgwidth\undefined%
    \setlength{\unitlength}{87.765625bp}%
    \ifx\svgscale\undefined%
      \relax%
    \else%
      \setlength{\unitlength}{\unitlength * \real{\svgscale}}%
    \fi%
  \else%
    \setlength{\unitlength}{\svgwidth}%
  \fi%
  \global\let\svgwidth\undefined%
  \global\let\svgscale\undefined%
  \makeatother%
  \begin{picture}(1,0.81188385)%
    \put(0,0){\includegraphics[width=\unitlength]{bousfield2simplexmesss.pdf}}%
    \put(0.46026641,0.00129077){\color[rgb]{0,0,0}\makebox(0,0)[lb]{\smash{0}}}%
    \put(0.95041837,0.68830897){\color[rgb]{0,0,0}\makebox(0,0)[lb]{\smash{1
}}}%
    \put(0.17562756,0.2781256){\color[rgb]{0,0,0}\makebox(0,0)[lb]{\smash{a}}}%
    \put(0.72253872,0.2781256){\color[rgb]{0,0,0}\makebox(0,0)[lb]{\smash{b}}}%
    \put(-0.00667616,0.68830897){\color[rgb]{0,0,0}\makebox(0,0)[lb]{\smash{2}}}%
    \put(0.38048794,0.74262984){\color[rgb]{0,0,0}\makebox(0,0)[lb]{\smash{[a,b]}}}%
  \end{picture}%
\endgroup%

%% file: bousfield3simplexNEW.pdf_tex
\begingroup%
  \makeatletter%
  \providecommand\color[2][]{%
    \errmessage{(Inkscape) Color is used for the text in Inkscape, but the package 'color.sty' is not loaded}%
    \renewcommand\color[2][]{}%
  }%
  \providecommand\transparent[1]{%
    \errmessage{(Inkscape) Transparency is used (non-zero) for the text in Inkscape, but the package 'transparent.sty' is not loaded}%
    \renewcommand\transparent[1]{}%
  }%
  \providecommand\rotatebox[2]{#2}%
  \ifx\svgwidth\undefined%
    \setlength{\unitlength}{101.45646973bp}%
    \ifx\svgscale\undefined%
      \relax%
    \else%
      \setlength{\unitlength}{\unitlength * \real{\svgscale}}%
    \fi%
  \else%
    \setlength{\unitlength}{\svgwidth}%
  \fi%
  \global\let\svgwidth\undefined%
  \global\let\svgscale\undefined%
  \makeatother%
  \begin{picture}(1,0.89884518)%
    \put(0,0){\includegraphics[width=\unitlength]{bousfield3simplexNEW.pdf}}%
    \put(0.25870974,0.00111663){\color[rgb]{0,0,0}\makebox(0,0)[lb]{\smash{0}}}%
    \put(0.95710906,0.28272938){\color[rgb]{0,0,0}\makebox(0,0)[lb]{\smash{1
}}}%
    \put(-0.00600628,0.37284558){\color[rgb]{0,0,0}\makebox(0,0)[lb]{\smash{3}}}%
    \put(0.44457398,0.84032256){\color[rgb]{0,0,0}\makebox(0,0)[lb]{\smash{2}}}%
    \put(0.10100652,0.1757167){\color[rgb]{0,0,0}\makebox(0,0)[lb]{\smash{a}}}%
    \put(0.43894165,0.49112266){\color[rgb]{0,0,0}\makebox(0,0)[lb]{\smash{b}}}%
    \put(0.70365768,0.13629092){\color[rgb]{0,0,0}\makebox(0,0)[lb]{\smash{c}}}%
    \put(0.43330933,0.28272938){\color[rgb]{1,0,0}\makebox(0,0)[lb]{\smash{[a,c]}}}%
    \put(0.72618673,0.59250326){\color[rgb]{1,0,0}\makebox(0,0)[lb]{\smash{[b,c]}}}%
  \end{picture}%
\endgroup%

%% file: bousfieldrelation.pdf_tex
\begingroup%
  \makeatletter%
  \providecommand\color[2][]{%
    \errmessage{(Inkscape) Color is used for the text in Inkscape, but the package 'color.sty' is not loaded}%
    \renewcommand\color[2][]{}%
  }%
  \providecommand\transparent[1]{%
    \errmessage{(Inkscape) Transparency is used (non-zero) for the text in Inkscape, but the package 'transparent.sty' is not loaded}%
    \renewcommand\transparent[1]{}%
  }%
  \providecommand\rotatebox[2]{#2}%
  \ifx\svgwidth\undefined%
    \setlength{\unitlength}{179.6796875bp}%
    \ifx\svgscale\undefined%
      \relax%
    \else%
      \setlength{\unitlength}{\unitlength * \real{\svgscale}}%
    \fi%
  \else%
    \setlength{\unitlength}{\svgwidth}%
  \fi%
  \global\let\svgwidth\undefined%
  \global\let\svgscale\undefined%
  \makeatother%
  \begin{picture}(1,0.400139)%
    \put(0,0){\includegraphics[width=\unitlength]{bousfieldrelation.pdf}}%
    \put(0.22468945,0.0028993){\color[rgb]{0,0,0}\makebox(0,0)[lb]{\smash{0}}}%
    \put(0.74047635,-0.00000001){\color[rgb]{0,0,0}\makebox(0,0)[lb]{\smash{1
}}}%
    \put(-0.00339146,0.31621578){\color[rgb]{0,0,0}\makebox(0,0)[lb]{\smash{3}}}%
    \put(0.46410713,0.31621578){\color[rgb]{0,0,0}\makebox(0,0)[lb]{\smash{2}}}%
    \put(0.08565589,0.13812108){\color[rgb]{0,0,0}\makebox(0,0)[lb]{\smash{a}}}%
    \put(0.35279794,0.13812108){\color[rgb]{0,0,0}\makebox(0,0)[lb]{\smash{b}}}%
    \put(0.55315448,0.13812108){\color[rgb]{0,0,0}\makebox(0,0)[lb]{\smash{[a,c]}}}%
    \put(0.8648202,0.13812108){\color[rgb]{0,0,0}\makebox(0,0)[lb]{\smash{[b,c]}}}%
    \put(0.5086308,0.31621578){\color[rgb]{0,0,0}\makebox(0,0)[lb]{\smash{3}}}%
    \put(0.97612939,0.31621578){\color[rgb]{0,0,0}\makebox(0,0)[lb]{\smash{2}}}%
    \put(0.18572118,0.36501097){\color[rgb]{0,0,0}\makebox(0,0)[lb]{\smash{[a,b]}}}%
    \put(0.63461971,0.36631144){\color[rgb]{0,0,0}\makebox(0,0)[lb]{\smash{[[a,c],[b,c]]}}}%
  \end{picture}%
\endgroup%